\documentclass[english]{article}
\usepackage[T1]{fontenc}
\usepackage[latin9]{inputenc}
\usepackage{geometry}
\usepackage{babel}
\usepackage{array}
\usepackage{verbatim}
\usepackage{float}
\usepackage{mathtools}
\usepackage{enumitem}
\usepackage{multirow}
\usepackage{amsthm,amsmath, mathtools}
\usepackage{stackrel}
\usepackage{graphicx}
\usepackage[square, sort, numbers]{natbib}

\usepackage{authblk}

\usepackage[unicode=true,
 bookmarks=true,bookmarksnumbered=false,bookmarksopen=false,
 breaklinks=false,pdfborder={0 0 1},backref=false,colorlinks=false]
 {hyperref}
\hypersetup{pdftitle={Recognizing generalized Petersen graphs in linear time},
 pdfauthor={Matjaž Krnc},
 pdfkeywords={linear algorithms, generalized Petersen graphs, local-global properties}}

\providecommand{\tabularnewline}{\\}
\floatstyle{ruled}
\newfloat{algorithm}{tbp}{loa}
\providecommand{\algorithmname}{Algorithm}
\floatname{algorithm}{\protect\algorithmname}

\theoremstyle{remark}
\newtheorem{rem}{\protect Remark}
\theoremstyle{plain}
\newtheorem{lem}{Lemma}
\newlist{casenv}{enumerate}{4}
\setlist[casenv]{leftmargin=*,align=left,widest={iiii}}
\setlist[casenv,1]{label={{\itshape\ Case} \arabic*.},ref=\arabic*}
\setlist[casenv,2]{label={{\itshape\ Case} \roman*.},ref=\roman*}
\setlist[casenv,3]{label={{\itshape\ Case\ \alph*.}},ref=\alph*}
\setlist[casenv,4]{label={{\itshape\ Case} \arabic*.},ref=\arabic*}
\theoremstyle{plain}
\newtheorem{cor}{Corollary}

\usepackage{hyperref}
\usepackage{amssymb,amsmath}
\usepackage[noend]{algpseudocode}

\title{Recognizing generalized Petersen graphs in linear time}

\author[1]{Matja\v z Krnc}
\author[2]{Robin J. Wilson}

\affil[1]{University of Primorska, Slovenia, and University of Salzburg, Austria. 
\protect\\ E-mail: \texttt{matjaz.krnc@upr.si}.
}
\affil[2]{The Open University, Milton Keynes, and 
 the London School of Economics, UK. 
 \protect\\ E-mail: \texttt{r.j.wilson@open.ac.uk}.
 }
\date{}
\begin{document}
\maketitle

\begin{abstract}
By identifying a local property which structurally classifies any edge, we show that the family of generalized Petersen graphs can be recognized in linear time.%
\end{abstract}

\bigskip{}
\global\long\def\gp#1{G\left(#1\right)}%
\global\long\def\P{\mathrm{GP}}%
\global\long\def\I#1{E_{I}\left(#1\right)}%
\global\long\def\O#1{E_{J}\left(#1\right)}%
\global\long\def\S#1{E_{S}\left(#1\right)}%
\global\long\def\nicefrac#1#2{#1/#2}%
\global\long\def\C{\mathcal{C}}%
The \emph{generalized Petersen graphs}, introduced by Coxeter \citep{coxeter1950self}
and named by Watkins \citep{watkins1969theorem}, are cubic graphs
formed by connecting the vertices of a regular polygon to the corresponding
vertices of a star polygon. Various aspects of their structure have
been extensively studied. Examples include identifying generalized
Petersen graphs that are Hamiltonian \citep{alspach1983classification,alspach1981result,castagna1972every},
hypo-Hamiltonian~\citep{bondy1972variations}, Cayley \citep{nedela1995generalized,saravzin1997note},
or partial cubes~\citep{klavvzar2003partial}, and finding their
automorphism group \citep{frucht1971groups} or determining isomorphic
members of the family \citep{petkovvsek2009enumeration}. Additional
aspects of the mentioned family are well surveyed in \citep{wilson92,holton1993petersen} while notable recent advances in the field can be found in \citep{,MR3906715,MR3991621,MR4040935,MR4041982,MR3878281,MR3967513, MR4014153}. 

\begin{figure}[bh]
\begin{centering}
\includegraphics[scale=0.6]{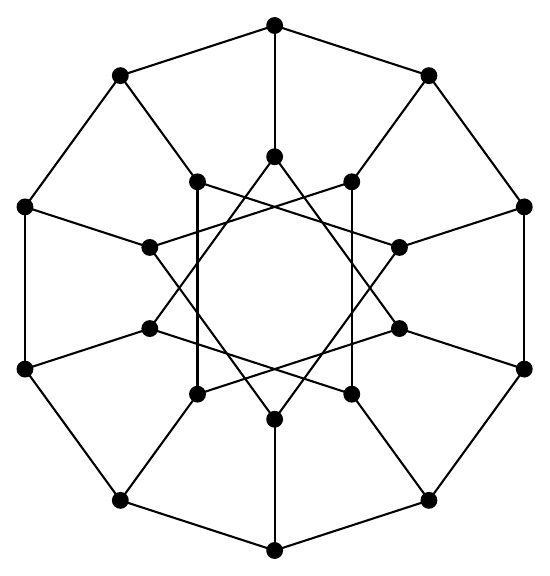}
\par\end{centering}
\caption{The generalized Petersen graph $G(10,3)$, also known as the Desargues
graph.}
\end{figure}

In this paper we give a linear-time recognition algorithm for the
family of generalized Petersen graphs. In particular, we identify
a local property which structurally classifies any edge, whenever
our graph is a generalized Petersen graph. We start by providing the
necessary definitions and the analysis of $8$-cycles; after the preliminaries
in Section~\ref{sec:Preliminaries} we describe $8$-cycles in Section~\ref{subsec:Description-of--cycles}.
In Section~\ref{sec:Recognizing-GP} we introduce our main lemma,
and describe and analyze the recognition procedure. At the end we
mention how our Algorithm~\ref{alg:1} behaves when the input graph
is small.

\section{Preliminaries\label{sec:Preliminaries}}

We follow the notations of Watkins \citep{watkins1969theorem} for
the family $\P$ of generalized Petersen graphs. For given integers
$n$ and $k<\nicefrac n2$, we define the generalized Petersen graph
$\gp{n,k}$ as the graph with vertex-set $\left\{ u_{0},u_{1},\dots,u_{n-1},w_{0},w_{1},\dots,w_{n-1}\right\} $
and edge-set $\left\{ u_{i}u_{i+1}\right\} \cup\left\{ u_{i}w_{i}\right\} \cup\left\{ w_{i}w_{i+k}\right\} $,
for $i=0,1,\dots,i-1$, where all subscripts are taken modulo $n$.
We partition the edges of $\gp{n,k}$ as
\begin{itemize}
\item the edges $\O G$ from the \emph{outer rim} (of type $u_{i}u_{i+1}$)
which form a cycle of length $n$;
\item the edges $\I G$ from the \emph{inner rims} (of type $w_{i}w_{i+k}$)
inducing $\mathrm{gcd}(n,k)$ cycles of length $\nicefrac n{\mathrm{gcd}(n,k)}$;
\item the \emph{spokes} $\S G$ (of type $u_{i}w_{i}$) which correspond to a perfect
matching in $\gp{n,k}$.
\end{itemize}
\label{def:sigma}For each edge $e\in E(G)$, let $\sigma_{G}(e)$
be the number of $8$-cycles containing $e$, and let $\mathcal{P}_{G}$
be a partition of the edge-set of $G$, corresponding to the values
of $\sigma_{G}\left(e\right)$, for $e\in E\left(G\right)$. The mapping
$\sigma$ plays a crucial role in the structural classification of
edges in Section~\ref{subsec:Description-of--cycles}; for example,
any spoke $e$ of the Petersen graph $\gp{5,2}$ has the value $\sigma_{\gp{5,2}}\left(e\right)=8$
(see Figure~\ref{fig:different-8-cycles}).
\begin{figure}[htb]
\begin{centering}
\includegraphics[height=3cm]{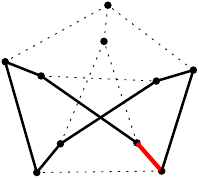}$\qquad$\includegraphics[height=3cm]{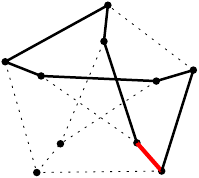}$\qquad$\includegraphics[height=3cm]{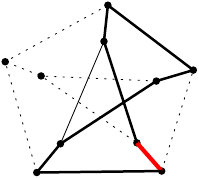}$\qquad$\includegraphics[height=3cm]{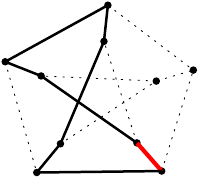}\bigskip{}
\includegraphics[height=3cm]{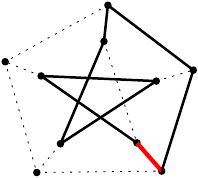}$\qquad$\includegraphics[height=3cm]{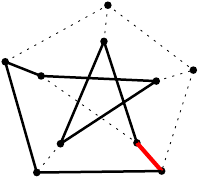}$\qquad$\includegraphics[height=3cm]{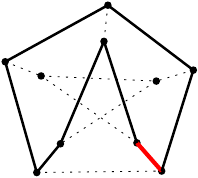}$\qquad$\includegraphics[height=3cm]{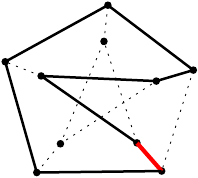}
\par\end{centering}
\caption{\label{fig:different-8-cycles}The Petersen graph $\protect\gp{5,2}$
admits 8-cycles of types $\protect\C_{2},\protect\C_{4}$, and $\protect\C_{7}$,
and any of its spokes $e$ has the value $\sigma_{\protect\gp{5,2}}\left(e\right)=8$.}
\end{figure}

\noindent The rotational symmetry of generalized Petersen graphs implies the following remark:
\begin{rem}
\label{rem:P}
The edges within any of the sets $\left\{ \O G,\S G,\I G\right\} $
share the same value of $\sigma_{G}$.
\end{rem}
For any $G\in\P$, we define $\sigma_{j}$ to be $\sigma_{G}(e)$,
where $e$ is a member of $\O G$; $\sigma_{s}$ and $\sigma_{i}$
are defined similarly. 

\begin{sloppypar}Given a graph $G$, we use the standard neighborhood
notations for vertices $N_{k}(v)=$ $\{ w\in V(G)\mid d_{G}(v,w)=k\} $
or edges $N_{k}(e)=\{ e'\in E(G)\mid d_{L(G)}(e,e')=k\} $,
where $L(G)$ is the line graph of $G$ and $d_{G}(v,w)$ is the distance
between $v$ and $w$ in $G$. The notation $G\left[S\right]$ corresponds
to the subgraph of $G$ induced by $S$, where $S\subseteq V(G)$
or $S\subseteq E(G)$. Graphs of order at most $80$ are considered
to be \emph{small }graphs, and can clearly be recognized in constant
time, so in what follows we assume that our graphs are not small --
that is, are \emph{large}. In this paper we study the large members
of $\P$ with respect to their values of $\sigma_{j},\sigma_{s},\sigma_{i}$,
and in particular determine that such a triple can correspond to only
one of eight possibilities, described below in (\ref{eq:triple-values}).
We proceed with the description of the possible $8$-cycles in $\gp{n,k}$.\end{sloppypar}
\section{A description of 8-cycles\label{subsec:Description-of--cycles}}

Given a graph $G\simeq\gp{n,k}$, fix a partitioning of $E(G)$ into
$\O G$, $\S G$, and $\I G$, as above. %
{} Define the rotation $\rho$ on $V\left(G\right)$ by 
$\rho\left(u_{i}\right)=u_{i+1}$
and $\rho\left(w_{i}\right)=w_{i+1}$, for $i\in\left[0,n-1\right]$.
Note that $\rho\in\mathrm{Aut}\left(G\right)$, so (non)edges are
mapped to (non)edges.

We call two $8$-cycles \emph{equivalent }whenever we can map one
into another by using rotations only. For a fixed $8$-cycle $C$,
we say that it has \emph{$c$-fold rotational symmetry} if there exist
precisely $c$ cycles equivalent to $C$ in $G$. Considering this
number $c$, and also the number of edges in $\O G,\S G,\I G$, we
can explicitly describe the contributions of $C$ to the global values
$\sigma_{j},\sigma_{s}$, and $\sigma_{i}$. These contributions uniquely
describe the \emph{type} of cycle $C$, and are denoted by $\delta_{j},\delta_{s}$,
and $\delta_{i}$, respectively. Whenever triples $\left(\delta_{j},\delta_{s},\delta_{i}\right)$
are used in this paper, the corresponding cycles are always explicitly
determined, or are clear from the context. We name eight specific
combinations of these triples $\left(\delta_{j},\delta_{s},\delta_{i}\right)$
in Table~\ref{tab:8cycles-contributions} below. For an example of
the above notions, see the eight cycles in Figure~\ref{fig:different-8-cycles},
which may be partitioned into three non-equivalent classes, corresponding
to the types $\C_{2},\C_{4}$, and $\C_{7}$.
\begin{table}[H]
\begin{centering}
\begin{tabular}{|c||c|c|c|c|c|c|c|c|}
\hline 
\textbf{label} & $\C_{0}$ & $\C_{1}$ & $\C_{2}$ & $\C_{3}$ & $\C_{4}$ & $\C_{5}$ & $\C_{6}$ & $\C_{7}$\tabularnewline
\hline 
\textbf{$\delta_{j}$} & 0 & 1 & 2 & 3 & 4 & 5 & 1 & 2\tabularnewline
\hline 
\textbf{$\delta_{s}$} & 0 & 2 & 2 & 2 & 2 & 2 & 2 & 4\tabularnewline
\hline 
\textbf{$\delta_{i}$} & 1 & 5 & 4 & 3 & 2 & 1 & 1 & 2\tabularnewline
\hline 
\end{tabular}
\par\end{centering}
\caption{\label{tab:8cycles-contributions}The labeling of all $8$ specific
types of $8$-cycles in $\protect\gp{n,k}$ is uniquely defined by
the corresponding values of $\left(\delta_{j},\delta_{s},\delta_{i}\right)$.
By characterizing $8$-cycles later we will see that no other type
of $8$-cycle may appear.}
\end{table}
If an $8$-cycle $C$ uses $s$ edges from $\S G$, we say that $s$
is the \emph{spoke value} of $C$. Note that the spoke value is always
$0,2$, or $4$, and that the spoke value coincides with $\delta_{s}$
only when the rotational symmetry of $C$ is $n$-fold. For the eight
specific triples of $\delta_{j},\delta_{s}$, and $\delta_{i}$ we
name the corresponding eight types of such cycles by $\left\{ \C_{i}\right\} _{i=0}^{7}$,
as described in Table~\ref{tab:8cycles-contributions}. We will see
later that no other types are possible for the generalized Petersen
graphs. 

We proceed by listing each of the possible $8$-cycles, giving its
spoke values $s$, and describing its rotational symmetries, the existence
conditions, and the corresponding values of $\delta_{j},\delta_{s}$,
and $\delta_{i}$. We summarize these results in Table~\ref{tab:8cycles}.
Watkins \citep{watkins1969theorem} has made a list of $8$-cycles
for generalized Petersen graphs, similar to our Table~\ref{tab:8cycles},
but he did not explain why his list is complete.

\begin{table}[H]
\begin{centering}
\begin{tabular}{|c|c|c|}
\hline 
type & a representative of an $8$-cycle & existence conditions\tabularnewline
\hline 
\hline 
$\C_{0}$ & $\left(w_{0},w_{k},\dots,w_{7k}\right)$ & $n=8k$ or $3n=8k$\tabularnewline
\hline 
$\C_{1}$ & $\left(u_{0},w_{0},w_{k},w_{2k},w_{3k},w_{4k},w_{5k},u_{5k}\right)$ & $n=5k\pm1$ or $2n=5k\pm1$\tabularnewline
\hline 
\multirow{2}{*}{$\C_{2}$} & $\left(u_{0},u_{1},u_{2},w_{2},w_{2+k},w_{2+2k},w_{2+3k},w_{2+4k}\right)$ & $n=4k+2$ or $2n=4k+2$\tabularnewline
\cline{2-3} 
 & $\left(u_{2},u_{1},u_{0},w_{0},w_{k},w_{2k},w_{3k},w_{4k}\right)$ & $k=\nicefrac{\left(n+2\right)}4$\tabularnewline
\hline 
\multirow{3}{*}{$\C_{3}$} & $\left(u_{0},u_{1},u_{2},u_{3},w_{3},w_{\left(\nicefrac n3\right)+2},w_{\left(\nicefrac{2n}3\right)+1},w_{0}\right)$ & $k=\left(\nicefrac n3\right)-1$\tabularnewline
\cline{2-3} 
 & $\left(u_{3},u_{2},u_{1},u_{0},w_{0},w_{\left(\nicefrac n3\right)+1},w_{\left(\nicefrac{2n}3\right)+2},w_{3}\right)$ & $k=\left(\nicefrac n3\right)+1$\tabularnewline
\cline{2-3} 
 & $\left(u_{0},u_{1},u_{2},u_{3},w_{3},w_{2},w_{1},w_{0}\right)$ & $k=1$\tabularnewline
\hline 
\multirow{2}{*}{$\C_{4}$} & $\left(u_{0},u_{1},\dots,u_{4},w_{4},w_{\nicefrac{\left(n+4\right)}2},w_{0}\right)$ & $k=\left(\nicefrac n2\right)-2$\tabularnewline
\cline{2-3} 
 & $\left(u_{0},u_{1},\dots,u_{4},w_{4},w_{2},w_{0}\right)$ & $k=2$\tabularnewline
\hline 
$\C_{5}$ & $\left(u_{0},u_{\pm1},\dots,u_{\pm5},w_{\pm5},w_{0}\right)$ & $k=5$ or $k=n-5$ \tabularnewline
\hline 
$\C_{6}$ & $\left(u_{0},u_{1},w_{1},w_{\nicefrac n2},u_{\nicefrac n2},u_{\nicefrac{\left(n+2\right)}2},w_{\nicefrac{\left(n+2\right)}2},w_{0}\right)$ & $k=\left(\nicefrac n2\right)-1$\tabularnewline
\hline 
$\C_{7}$ & $\left(u_{0},u_{1},w_{1},w_{k+1},u_{k+1},u_{k},w_{k},w_{0}\right)$ & $n\geq4$\tabularnewline
\hline 
\end{tabular}
\par\end{centering}
\caption{\label{tab:8cycles}Characterization of all $8$-cycles from $\protect\gp{n,k}$.}
\end{table}
\noindent We now consider the values of $s$ in turn.

\subsection*{$8$-cycles with $s=0$.\medskip{}
}

\noindent We first note that, since $G$ is large, the cycle in the
outer rim cannot be of length $8$, and so the edge-set of any such
cycle is a subset of $\I G$, implying the necessary condition $\nicefrac n{\gcd\left(n,k\right)}=8$.
It follows that either there are no $8$-cycles with $s=0$, or there
are $\nicefrac n8$ such cycles, with $k=\nicefrac n8$ or $\nicefrac{3n}8$.
We label these cycles $\C_{0}$. Also observe that these cycles (if
they exist) have $\left(\nicefrac n8\right)$-fold rotational symmetry,
so for each edge $e\in\I G$, we have $\sigma(e)=1$.

\subsection*{$8$-cycles with $s=2$.\medskip{}
}

\noindent We next focus on the $8$-cycles that contain two edges
from $\S G$. Any such $8$-cycle contains $j$ edges in $\O G$ and
$6-j$ edges in $\I G$, and may be rotated by the $n$-fold rotational
symmetry of $G$. We further observe that $1\leq j\leq5$, and label
these $8$-cycles $\C_{j}$. It follows that $\left(\delta_{j},\delta_{s},\delta_{i}\right)=\left(j,2,6-j\right)$.
We now consider all possibilities with respect to their value $j$.
\begin{description}
\item [{$\C_{1}$}] When $j=1$, any corresponding cycle is of type $\left(u_{0},w_{0},w_{k},w_{2k},w_{3k},w_{4k},w_{5k},u_{5k}\right)$
if and only if $5k\pm1\equiv0\pmod n$.
\item [{$\C_{2}$}] Similarly, if such an $8$-cycle exists, then $4k\equiv\pm\:2\pmod n$,
and since $k<\nicefrac n2$, it follows that $k=\nicefrac{\left(n\pm2\right)}4$
or $k=\nicefrac{\left(n-1\right)}2$. In the former case $n\equiv2\pmod4$,
and in the latter case $n$ is odd.
\item [{$\C_{3}$}] If $j=3$, then $3k\equiv\pm3\pmod n$. So $3$ divides
$n$, and we have $k=1$ or $\left(\nicefrac n3\right)-1$ or $\left(\nicefrac n3\right)+1$.
\item [{$\C_{4}$}] If $j=4$, then $2k\equiv\pm4\pmod n$. Since $k<\nicefrac n2$,
there are again two possibilities: $k=2$ or $k=\nicefrac{\left(n-4\right)}2$. 
\item [{$\C_{5}$}] If $j=5$, then $k\equiv5\pmod n$.
\end{description}

\subsection*{$8$-cycles with $s=4$.\medskip{}
}

\noindent First, we consider a special type of $8$-cycle of the form
\[
\left(u_{0},u_{1},w_{1},w_{\nicefrac n2},u_{\nicefrac n2},u_{\nicefrac{\left(n+2\right)}2},w_{\nicefrac{\left(n+2\right)}2},w_{0}\right),
\]
whenever $n=2k+2$. These $8$-cycles have only $\left(\nicefrac n2\right)$-fold
rotational symmetry, so $\delta_{j}=1$, $\delta_{s}=2$, and $\delta_{i}=1$
and are labeled by $\C_{6}$.

Finally, consider the $8$-cycles with four edges from $\S G$ of
type 
\[
\left(u_{i},w_{i},w_{i+k},u_{i+k},u_{i+k+1},w_{i+k+1},w_{i+1},u_{i+1}\right),
\]
where $0\leq i\leq n-1$.  Since $G$ is assumed to be large, these
cycles are always present. Since they have $n$-fold symmetry, $\delta_{j}=2,\delta_{s}=4$,
and $\delta_{i}=2$ and are denoted by $\C_{7}$. 

\section{Recognizing generalized Petersen graphs\label{sec:Recognizing-GP}}

\noindent Using the structure of $8$-cycles in $G$, we have the
following property.
\begin{lem}
\label{lem:part-of-size-n}Let $\gp{n,k}\in\P$ be a large graph,
and let $\mathcal{P}$ be a partition of its edge-set, assigning edges
$e$ with the same value of $\sigma\left(e\right)$ to the same partition.
Then $\mathcal{P}$ contains a part of size $n$.
\end{lem}
\begin{proof}
Let $G=\gp{n,k}$. By Remark~\ref{rem:P}, it is enough to prove
that $\left|\mathcal{P}\right|>1$. In addition to the cycles of type
$\C_{7}$, more than one additional type of $8$-cycle may coexist
in $G$ (see existence conditions on Table~\ref{tab:8cycles}). 

We show that this can never happen in a large graph. Indeed, assume
that, together with $n$ cycles of type $\C_{7}$, two distinct additional
types of $8$-cycles exist and consider the possible $\binom{7}{2}=21$
cases. 
\begin{casenv}
\item[\emph{Case 1.}] There exist cycles of type $\C_{7},\C_{0},$ and $\C_{1}$. Then $n\geq4$,
and $n=8k$ or $3n=8k$, and $n=5k\pm1$ or $2n=5k\pm1$. But then
our graph is small.
\item[\emph{Case 2.}] There exist cycles of type $\C_{7},\C_{0},$ and $\C_{2}$. Then $n\geq4$,
and $n=8k$ or $3n=8k$, and $n=4k+2$ or $2n=4k+2$ or $k=\nicefrac{\left(n+2\right)}4$.
Again, this implies that our graph is small.
\[
\hdots
\]
\item[\emph{Case 5.}]  There exist cycles of type $\C_{7},\C_{0},$ and $\C_{1}$. Then
$n\geq4$, and $n=8k$ or $3n=8k$, and $k=5$ or $k=n-5$. Then we
could have $G\simeq\gp{40,5}$. Again our graph is small.
\[
\hdots
\]
\item[\emph{Case 21.}] There exist cycles of type $\C_{7},\C_{5},$ and $\C_{6}$. Then
$n\geq4$, and $k=5$ or $k=n-5$, and $k=\left(\nicefrac n2\right)-1$.
But then either $G\simeq\gp{12,6}$ or $G\simeq\gp{8,3}$, which is
still small.
\end{casenv}
There are at most eight distinct possibilities for the corresponding
values of $\left(\sigma_{j},\sigma_{s},\sigma_{i}\right)$. In particular,
\begin{align}
\left(\sigma_{j},\sigma_{s},\sigma_{i}\right) & \in\left\{ \left(2,4,3\right)\!,\left(3,6,7\right)\!,\left(4,6,6\right)\!,\left(5,6,5\right)\!,\left(6,6,4\right)\!,\left(7,6,3\right)\!,\left(3,6,3\right)\!,\left(2,4,2\right)\right\} ,\label{eq:triple-values}
\end{align}
 where the values above are obtained by adding $\left(2,4,2\right)$
to the possible values of $\left(\delta_{j},\delta_{s},\delta_{i}\right)$
-- see Table~\ref{tab:8cycles-contributions}. To prove this, it
is enough to observe that in the eight possible triples from (\ref{eq:triple-values}),
the values never coincide.
\end{proof}

\subsection{A recognition algorithm}

\noindent In this section we describe a simple procedure \texttt{main$(G)$},
which runs in time $O(n)$ and determines whether the input graph
$G$ is a member of $\P$. It is described in Algorithm~\ref{alg:1},
and uses an additional procedure \texttt{extend($G,U$)} that is given
in Algorithm~\ref{alg:Extend(G,U)}.

The tasks of these Algorithms~\ref{alg:Extend(G,U)}~and~\ref{alg:1}
basically correspond to identifying the vertices of the outer rim,
and checking whether this outer rim can be extended to a proper member
of $\P$. 

The input of Algorithm~\ref{alg:Extend(G,U)} consists of a connected
cubic graph $G$ of order $2n$, and the $n$-set $U$ of vertices.
A procedure \texttt{extend($G,U$)} assumes that $U$ corresponds
to either $\O G$ or $\I G$, and requires $O(n)$ time to decide
whether $G\in\P$. Indeed, once one identifies the vertices in $\O G$
by using either $U$ or its complement (up to line~\ref{alg0:set-U}),
a bijection $V(G)\rightarrow\left\{ u_{0},u_{1},\dots,u_{n-1},w_{0},w_{1},\dots,w_{n-1}\right\} $
is established and $k$ is easily determined (see line~\ref{alg0:k}).
So it is enough to check that the edges of $G$ do indeed map to the
edges of $\gp{n,k}$. 

Algorithm~\ref{alg:1} basically consists of:
\begin{enumerate}[label=\roman*.]
\item \label{enu:i} determining the values of $\sigma_{G}(e)$ for each
edge $e\in E(G)$;
\item \label{enu:ii}determining an $n$-subset of $E(G)$ which is also
a member of $\left\{ \O G,\S G,\I G\right\} $, whenever $G\in\P$; 
\item \label{enu:iii} identifying a vertex-set $U$ which is one of $\left\{ u_{i}\right\} _{i=0}^{n-1}$
or $\left\{ w_{i}\right\} _{i=0}^{n-1}$, and running \texttt{extend($G,U$)}
accordingly. 
\end{enumerate}
We comment on these three statements in turn. 
\begin{enumerate}
\item[\ref{enu:i}]  For any $e\in E(G)$, all $8$-cycles that contain $e$ lie within
its $4$-neighborhood. If $G'=G\left[\cup_{i=0}^{4}N_{i}\left(e\right)\right]$,
then $\sigma_{G}(e)=\sigma_{G'}(e)$. But since $G$ is cubic, the
order of $G'$ is bounded above by 62. So the calculation of $\sigma_{G}(e)$
takes $O(1)$ time, and the whole loop at line~\ref{alg1:loop} altogether
takes at most $O\left(\left|E(G)\right|\right)=O(n)$ time. In line~\ref{alg1:edge-classification},
according to its value of $\sigma_{G}$, each edge is classified to
the corresponding part from an edge partition $\mathcal{P}$.
\item[\ref{enu:ii}]  While line~\ref{alg1:(ii)} is trivially of constant time-complexity,
its correctness is provided by Lemma~\ref{lem:part-of-size-n}. Note
that if $\min_{U\in\mathcal{P}}\left|U\right|\neq n$, then the same
lemma allows us to return \texttt{\textbf{False}}.
\item[\ref{enu:iii}]  If $G$ is a generalized Petersen graph, it is enough to determine
a 2-factor in $G$ of order $n$, and assign it to $U$. This is trivially
satisfied if the graph selected in line~\ref{alg1:(ii)} already
corresponds to either $\O G$ or $\I G$. If $U$ (as selected in
line~\ref{alg1:(ii)}) corresponds to $\S G$, then it is a perfect
matching of order $2n$ and the largest cycle of $G-E(U)$ clearly
corresponds to $\O G$. 
\end{enumerate}

\subsection{Recognizing small graphs}

\noindent It is clear that for graphs on at most $80$ vertices, the
membership of $\P$ may theoretically be determined in a constant
time. However, the task may not be easy in practice. We now argue
that Algorithm~\ref{alg:1} fails for precisely ten members of $\P$,
and that it can safely be used on any graphs of order \emph{not }equal
to $6,8,10,16,20,24,26,48,$ or $52$.

It is clear that Lemma~\ref{lem:part-of-size-n} cannot hold for
all small graphs. Indeed, by Frucht \emph{et al.} \citep{frucht1971groups},
we know that there exist precisely seven pairs $(n,k)$ for which
$\gp{n,k}$ is edge-transitive: these are $(n,k)=(4,1),(5,2),(8,3),(10,2),(10,3),(12,5)$,
and $(24,5)$. In these cases all edges have the same value for $\sigma$. 

Using a computer one can easily calculate the $\sigma$-partitions
for the remaining 373 $(n,k)$-pairs which correspond to small generalized
Petersen graphs. We have checked whether the corresponding $\sigma$-partition
of edges consists of more than one part and have determined that,
in addition to the above seven pairs, there are three additional members
of $\P$ for which all edges touch the same number of $8$-cycles.
These three additional cases are $\gp{3,1},\gp{13,5}$, and $\gp{26,5}$.
For $\gp{3,1}$, we trivially have $\sigma(e)=0$ for each edge $e$.
The remaining two cases contain $8$-cycles of types $C_{1},C_{5}$,
and $C_{7}$, and in these cases we have 
\begin{align*}
\left(\sigma_{j},\sigma_{s},\sigma_{i}\right) & =\left(1,2,5\right)+\left(5,2,1\right)+\left(2,4,2\right)=\left(8,8,8\right)\!,
\end{align*}
so $\sigma(e)=8$ for each edge $e$. It is worth noting that the
three exceptional pairs correspond to vertex-transitive members of
$\P$, which gives the following corollary.
\begin{cor}
In the family of generalized Petersen graphs, there exist precisely
ten members $\gp{n,k}$ such that all edges touch the same number
of eight-cycles. These are all vertex-transitive, in particular
\[
(n,k)\in\left\{ (3,1),(4,1),(5,2),(8,3),(10,2),(10,3),(13,5),(12,5),(24,5),(26,5)\right\}\!.
\]
\end{cor}
It is clear that our algorithms will never recognize these ten special
graphs, but will safely recognize any graph on $n$ vertices, where
$n\notin\left\{ 6,8,10,16,20,24,26,48,52\right\} $, in $O(n)$ time.
\medskip{}

\noindent \textbf{Acknowledgement}

\noindent   The authors would like to thank Prof. Martin Milani\v{c}
and the anonymous reviewers for their valuable comments during the
preparation of this paper.
The first author gratefully acknowledges 
the Slovenian Research Agency for funding project P1-0383 and program J1-1692, and 
the European Commission for funding the InnoRenew CoE project (Grant Agreement \#739574) under the Horizon2020 Widespread-Teaming program and the Republic of Slovenia (Investment funding of the Republic of Slovenia and the European Union of the European Regional Development Fund). 
\begin{algorithm}[p]
\begin{algorithmic}[1]
\Require{a cubic non-labeled graph $G$ on $2n$ vertices, and a set $U\subseteq V(G)$}
\If {$G\left[U\right]\not\simeq C_{n}$}
	\If {$|U|\neq n$ or $G\left[V(G)\setminus U\right]\not\simeq C_{n}$}
		\State return False
	\Else 
		\State $U\gets V(G)\setminus U$
	\EndIf
\EndIf
\State label $U=\left[u_{0},u_{1},\dots,u_{n-1}\right]$ cyclically w.r.t. $C_n$ \label{alg0:set-U}
\For {$0\leq i<n$}
	\State label the unique vertex in $N(u_i)\setminus U$ by $w_i$
\EndFor
\State $k\gets 1$
\While {$w_0w_{k-1}\in E(G)$}
	\State $k\gets k+1$ \label{alg0:k}
\EndWhile
\For {$1\leq i<n$}
	\If {$w_i w_{i+k} \notin E(G)$}
		\State return False
	\EndIf
\EndFor
\State return True
\end{algorithmic}
\caption{\label{alg:Extend(G,U)}Extend $(G,U)$}
\end{algorithm}
\begin{algorithm}[p]
\begin{algorithmic}[1]
\Require{a cubic graph $G$}
\State initialize $\mathcal{P}$ of type $\mathbb{N}\rightarrow 2^{E(G)}$
\For {$e\in E(G)$} \label{alg1:loop}
	\State $G'\gets G\left[\cup_{i=0}^4N_i(v)\right]$
	\State $s_e\gets \sigma_{G'}(e)$
	\State $\mathcal{P}(s_e)\gets \mathcal{P}(s_e)\cup \left\{e\right\}$ \label{alg1:edge-classification}
\EndFor
\State $U\gets$ a set of edges from Im$(\mathcal{P})$ with minimal positive cardinality \label{alg1:(ii)}
\If {$|U| \neq n$}
	\State return False
\EndIf
\If {$G[U]$ has $2n$ vertices} \Comment{$U$ is a perfect matching in $G$}
	\State $U\gets $ edges from largest component of $G-E(U)$
\EndIf
\If {$G\left[U\right]$ is a $2$-factor}
	\State return $\mathtt{extend}(G,U)$
\Else
	\State return False
\EndIf
\end{algorithmic}
\caption{\label{alg:1}The recognition procedure.}
\end{algorithm}

\newpage
\bibliographystyle{amcjoucc}
\bibliography{kronecker}
\end{document}